\documentclass[11pt]{amsart}

\usepackage[margin=1.35 in]{geometry}

\usepackage{amsmath, amssymb, amsthm}
\usepackage[pdftex]{color}
\usepackage{comment}
\usepackage[colorlinks,citecolor=blue,linkcolor=red]{hyperref}
\usepackage{xcolor}
\usepackage{marginnote}
\usepackage{hyperref}

\usepackage[all,arc]{xy}

\newtheorem{theorem}{Theorem}
\newtheorem{lemma}[theorem]{Lemma}
\newtheorem{proposition}[theorem]{Proposition}

\theoremstyle{remark}

\theoremstyle{definition}

\newtheorem{remark}[theorem]{Remark}

\newcommand{\ev}{\mathrm{ev}}

\newcommand{\e}{\mathrm{ev}}

\newcommand{\nc}{\newcommand}
\nc{\dmo}{\DeclareMathOperator}

\nc{\R}{\mathbb{R}}
\nc{\Z}{\mathbb{Z}}
\nc{\N}{\mathbb{N}}
\nc{\cS}{\mathcal{S}}
\nc{\iso}{\cong}
\dmo{\Diff}{Diff}
\dmo{\Homeo}{Homeo}
\dmo{\dist}{dist}
\dmo\BDiff{BDiff}
\dmo\SO{SO}
\dmo\slide{sl}
\dmo\im{im}
\dmo\id{id}
\dmo\Fix{Fix}
\dmo\Out{Out}
\dmo{\T}{\mathcal{T}}
\dmo{\Te}{\mathcal{T}^{\epsilon}}
\dmo{\Me}{\mathcal{M}^{\epsilon}}

\begin{document}

\title[A central limit theorem for random closed geodesics]{A central limit theorem for random closed geodesics: proof of the Chas--Li--Maskit conjecture}

\author[I. Gekhtman]{Ilya Gekhtman}
\address{Department of Mathematics\\ 
University of Toronto\\ 
40 St George St\\ 
Toronto, ON, Canada\\}
\email{\href{mailto:ilyagekh@gmail.com}{ilyagekh@gmail.com}}

\author[S.J. Taylor]{Samuel J. Taylor}
\address{Department of Mathematics\\ 
Temple University\\ 
1805 North Broad Street\ 
Philadelphia, PA 19122, U.S.A\\}
\email{\href{mailto:samuel.taylor@temple.edu}{samuel.taylor@temple.edu}}

\author[G. Tiozzo]{Giulio Tiozzo}
\address{Department of Mathematics\\ 
University of Toronto\\ 
40 St George St\\ 
Toronto, ON, Canada\\}
\email{\href{mailto:tiozzo@math.toronto.edu}{tiozzo@math.toronto.edu}}

\date{\today}
\thanks{Gekhtman is partially supported by NSF grant DMS-1401875.\\
\indent Taylor is partially supported by NSF grant DMS-1744551. \\
\indent Tiozzo is partially supported by NSERC and the Alfred P. Sloan Foundation.
}

\begin{abstract}
We prove a central limit theorem for the length of closed geodesics in any compact orientable hyperbolic surface. In the special case of a hyperbolic pair of pants, this 
settles a conjecture of Chas--Li--Maskit.
\end{abstract}

\maketitle

\section{Introduction}
Let $\Sigma$ be a compact orientable hyperbolic surface whose boundary, if any, is geodesic, and let $G$ denote its fundamental group. 
By a standard generating set $S$ for $G$ we mean the following: 
when $G$ is free (i.e. when $\partial \Sigma \neq \emptyset$) $S$ is a free basis for $G$. Otherwise, $\Sigma$ is a closed orientable surface of genus $g \ge 2$ and $S$ is the generating set used in the usual presentation $G = \langle a_1,\ldots,a_g,b_1, \ldots b_g : \prod_i [a_i,b_i] = 1 \rangle$.

Now fix a standard generating set $S$ of $G$, and let $|g |$ be the word length of $g$ with respect to $S$. 
For each $g \in G$, let $[g]$ denote its conjugacy class, and for any conjugacy class $\gamma = [g]$ define its conjugacy length 
$\Vert \gamma \Vert = \Vert g \Vert := \min_{[g] = \gamma} | g |$ to 
be the minimum word length over all elements representing $\gamma$.

Any conjugacy class $\gamma$ is represented by a closed geodesic in $\Sigma$, and let $\tau(\gamma)$ denote the length of this geodesic in the hyperbolic metric. 
Let $\mu_n$ denote the uniform distribution on the set $\mathcal{F}_n$ of conjugacy classes of length $n$. 
The goal of this note is to prove the following central limit theorem: 

\begin{theorem} \label{T:main}
There exist constants $L > 0$, $\sigma > 0$ such that for any $a, b \in \mathbb{R}$ with $a < b$ we have 
$$\mu_n\left(\gamma \ : \  \frac{\tau(\gamma) - n L}{\sigma \sqrt{n}} \in [a, b] \right) \to \frac{1}{\sqrt{2 \pi}} \int_a^b e^{-\frac{x^2}{2}} \ dx$$
as $n \to \infty$. 
\end{theorem}

Motivated by experimental evidence, Chas--Li--Maskit \cite{CLM} conjectured that the conclusion of Theorem \ref{T:main} holds for a hyperbolic pair of pants. 
This followed an earlier central limit theorem by Chas--Lalley \cite{CL} for the distribution of self-intersection numbers of random geodesics. 

The proof of Theorem \ref{T:main} uses the central limit theorem for H\"older continuous observables on a mixing Markov chain following Ruelle \cite{Ruelle} and Bowen \cite{Bowen} 
(see also Pollicott--Sharp \cite{PS} and Calegari \cite{Ca}), combined with estimates on Gromov products by the authors \cite{GTT}. The Markov chain which encodes closed geodesics on a surface is provided by Series \cite{Se} (see also Wroten \cite{Wr}).


\smallskip
We note that P. Park has recently written up a related result where the uniform distribution on conjugacy classes is replaced by the $n^{th}$  step distribution of a 
simple random walk on $G$ \cite{Park}. Let us note that counting for the simple random walk and counting with respect to balls in the Cayley graph are in general different, and many authors addressed the question of how they are related on various groups. 

\section{Preliminaries}

\subsection*{The geometric setup}
Throughout we consider $G$ as a discrete group of isometries of the hyperbolic plane $\mathbb{H}^2$. When $\Sigma$ is closed, $G$ is a cocompact Fuchsian group. Otherwise, $\partial \Sigma \neq \emptyset$ and $G$ is the free group $F_N$ for some $N \ge 2$.
In this case, we have that $\Sigma$ is the convex core of $\mathbb{H}^2/G$, where $G$ acts on $\mathbb{H}^2$ as a Schottky group.

Fix a base point $z \in \mathbb{H}^2$. 
Then for $\gamma = [g]$, $\tau(\gamma) = \tau(g)$ equals the (stable) translation length of $g$ on $\mathbb{H}^2$. Hence, one has the formula (see e.g. \cite[Proposition 5.8]{MT})
\begin{equation} \label{E:tau}
\tau(g) = d(gz, z) - 2 (gz, g^{-1} z)_z + O(\delta)
\end{equation}
where $(x, y)_z$ denotes the Gromov product. Here $A = B + O(\delta)$ means that there is a constant $C$, which depends only on the hyperbolicity constant $\delta$ of $\mathbb{H}^2$, such that $|A - B| \leq C$.

\subsection*{Some basic probability}

We begin by recording a few basic lemmas that will be needed for our arguments.
\begin{lemma} \label{L:basic}
Let $(A_n)$ be any sequence of measurable sets in a probability space, $(\mathbb{P}_n)$ a sequence of probability measures, and let $(B_n)$ a sequence such that 
$$\lim_{n \to \infty} \mathbb{P}_n(B_n) = 1.$$
Then 
$$\limsup_n | \mathbb{P}_n(A_n) - \mathbb{P}_n(A_n \cap B_n) | = 0.$$
\end{lemma}

\begin{proof}
By elementary set theory,
$$\mathbb{P}_n(A_n) = \mathbb{P}_n(A_n \cap B_n) + \mathbb{P}_n(A_n \setminus B_n) \leq \mathbb{P}_n(A_n \cap B_n) + \mathbb{P}_n(B_n^c)$$
which yields the claim.
\end{proof}

\begin{lemma} \label{L:sum}
Let $(\mathbb{P}_n)$ be a sequence of probability measures on a Borel space $X$, and let $F, G : X \to \mathbb{R}$ be two measurable functions. 
Suppose that $\mathbb{P}_n(|G(x)| \geq \epsilon) \to 0$ for any $\epsilon > 0$, and let $(c_n)$ be a sequence of positive real numbers with $\lim_{n \to \infty} c_n = 1$. 
Suppose that there exists a continuous function $\rho : \mathbb{R} \to \mathbb{R}^+$ such that 
$$\lim_{n \to \infty} \mathbb{P}_n(F(x) \in [a, b]) = \int_a^b \rho(x) \ dx.$$
Then 
$$\lim_{n \to \infty} \mathbb{P}_n \left( \frac{F(x) + G(x)}{c_n} \in [a, b] \right)  = \int_a^b \rho(x) \ dx.$$
\end{lemma}

\begin{proof}
Fix $\epsilon > 0$, and denote $\Phi_{a, b} := \int_a^b \rho(x) \ dx$. If $n$ is sufficiently large, then by setting $Y_n := \frac{F + G}{c_n}$ we have
$$\{ F(x) \in [a+\epsilon, b - \epsilon] \textup{ and } |G(x)| \leq \epsilon/2 \} \subseteq \{Y_n \in [a, b] \}$$
hence using Lemma \ref{L:basic}
$$\liminf \mathbb{P}_n( Y_n \in [a, b] ) \geq \liminf \mathbb{P}_n(
F \in [a+\epsilon, b - \epsilon] \textup{ and } |G| \leq \epsilon/2) = $$
$$= \liminf \mathbb{P}_n(F \in [a+\epsilon, b - \epsilon]) = \Phi_{a+\epsilon, b - \epsilon}.$$
On the other hand 
$$\{ Y_n \in [a, b] \textup{ and } |G| \leq \epsilon/2 \} \subseteq \{F \in [a - \epsilon, b + \epsilon] \}$$
hence 
$$\limsup \mathbb{P}_n( Y_n \in [a, b] )= \limsup \mathbb{P}_n(
Y_n \in [a, b] \textup{ and } |G| \leq \epsilon/2) \leq $$
$$\leq \limsup \mathbb{P}_n(F \in [a- \epsilon, b + \epsilon]) = \Phi_{a-\epsilon, b + \epsilon}$$
and taking $\epsilon \to 0$ completes the proof. 
\end{proof}

\begin{lemma} \label{L:TV}
If $\lambda, \nu$ are purely atomic probability measures on a set $X$ and $\lambda$ is absolutely continuous with respect to $\nu$, then 
$$\Vert \lambda - \nu \Vert_{TV} \leq \left \Vert \frac{d \lambda}{d\nu} - 1 \right\Vert_\infty,$$
where $\Vert \cdot \Vert_{TV}$ denotes the total variation of a measure.
\end{lemma}

\begin{proof}
Since the measures are atomic, 
$$\frac{d\lambda}{d\nu}(x) = \frac{\lambda(x)}{\nu(x)} \qquad \textup{ for any }x \in X.$$
Then
$$\Vert \lambda - \nu \Vert_{TV} = \sup_{A \subseteq X} |\lambda(A)-\nu(A)| \leq \sup_A \sum_{x \in A} |\lambda(x) - \nu(x)| \leq $$
$$\leq \sup_A \sum_{x \in A} \left| \frac{\lambda(x)}{\nu(x)} - 1 \right| \nu(x) \leq \sup_A \left( \left\Vert \frac{d\lambda}{d\nu} - 1 \right\Vert_{\infty} \nu(A) \right)\leq 
\left\Vert \frac{d\lambda}{d\nu} - 1 \right\Vert_{\infty},$$
completing the proof.
\end{proof}

\section{The central limit theorem for displacement}

The proof of Theorem \ref{T:main} uses equation \eqref{E:tau}: basically, one proves a CLT for the displacement function $d(z, gz)$, and then shows that the contribution of the term $2 (gz, g^{-1}z)_z$ tends to zero. These two facts will suffice by Lemma \ref{L:sum}.
We will start by establishing the CLT for displacement.

\subsection*{Coding for closed geodesics}

First of all, we use that conjugacy classes in free and surface groups can be encoded by a finite graph. 

\begin{lemma} \label{L:code}
Let $G = \pi_1(\Sigma)$ where $\Sigma$ is an orientable hyperbolic surface of finite type. 
Let $S$ be a standard generating set for $G$. Then there exists an oriented graph $\Gamma$ whose edges are labeled by elements of $S \cup S^{-1}$ and 
such that: 
\begin{enumerate}
\item cycles in $\Gamma$ of length $n$ are in bijection with conjugacy classes $C(G)$ of length $n$ in the group, except for finitely many exceptions;
\item a conjugacy class is primitive if and only if the corresponding cycle in the graph is primitive, i.e. not the power of a shorter cycle; 
\item the adjacency matrix $M$ for this graph is aperiodic.
\end{enumerate}
\end{lemma}

For closed surface groups, Lemma \ref{L:code} is precisely (\cite{PS3}, Lemma 1.1), which is a reformulation of results of Series \cite{Se}.
Another coding for closed surfaces is given by Wroten \cite{Wr} and also yields this result. The much easier case of free groups is briefly explained below.

Recall that a matrix $M$ is \emph{aperiodic} if there exists an integer $k \geq 1$ such that all entries of $M^k$ are positive. 
By the Perron-Frobenius theorem, the matrix $M$ has a unique, simple eigenvalue $\lambda > 1$ of maximum modulus.
Also, if $\ev$ denotes the map which reads the labels off of oriented edges of $\Gamma$, then $\ev$ extends to the \emph{evaluation map} from directed paths in $\Gamma$ to $G$. In details, if a directed path $p$ is a concatenation of orientated edges $e_1,\ldots, e_n$, then $\ev(p) = \ev(e_1) \cdots \ev(e_n)$ in $G$.
It is this map that induces the bijection in item $(1)$ of Lemma \ref{L:code}. We further remark that Lemma \ref{L:code} implies that directed paths in $\Gamma$ map to geodesics in $G$; that is, if $p$ is a directed path of length $n$ in $\Gamma$, then $|\ev(p)| = n$ with respect to the generating set $S$.

We note that for free groups, the construction of the graph $\Gamma$ is immediate. 
Let $G = F_N$ and fix a basis $\{a_1, \dots, a_N\}$ of $F_N$. Then the graph $\Gamma$ has $2N$ vertices, labelled $a_i^\epsilon$ with $i = 1, \dots, N$, $\epsilon = \pm$. For each vertex $v = a_i^\epsilon$, there exists an edge labelled $a_j^\eta$ to the vertex $a_j^\eta$ unless $i = j$ and $\epsilon = - \eta$. In this case, nontrivial cyclically reduced words are in bijection with oriented based closed paths since such a word can only be read as a \emph{closed} path by starting at the vertex corresponding to the label on its last letter. As two cyclically reduced words represent conjugate elements if and only if they differ by cyclic permutation, which corresponds to changing the basepoint of the loop, this establishes items $(1)$ and $(2)$. (In this case, the only exception in item $(1)$ is for the trivial conjugacy class, which can be read as the trivial cycle based at each vertex.)
Item $(3)$ is clear from the construction.

\medskip

Let $\lambda_n$ be the uniform distribution on the set $\mathcal{C}_n$ of (based) closed paths of length $n$ in $\Gamma$. 
A short counting argument shows the following, which is a variation of (\cite{CL}, Lemma 4.1):
\begin{lemma} \label{L:push}
Let $p_n \colon \mathcal{C}_n \to \mathcal{F}_n$ be the map from closed paths to conjugacy classes induced by the evaluation map. Then 
$$\Vert (p_n)_\star \lambda_n  - \mu_n \Vert_{TV} \to 0\qquad \textup{as }n \to \infty,$$
where $\Vert \cdot \Vert_{TV}$ denotes the total variation of a measure. 
\end{lemma}

\begin{proof}
Let $\Gamma_n$ denote 
the set of primitive cycles of length $n$ (without remembering the basepoint). Note that $\textup{Tr }M^n = \lambda^n + \sum_{i =1}^k \lambda_i^n$, where $|\lambda_i| < \lambda$.
Hence $\lambda^n \leq \textup{Tr }M^n \leq c \lambda^n$ where $c$ only depends on $M$, and $\lim_{n \to \infty}\frac{\textup{Tr }M^n}{\lambda^n} = 1$. 
Moreover, by definition $\sum_{d | n} d (\# \Gamma_d)  =\textup{Tr }M^n$, so $n (\# \Gamma_n) \leq \textup{Tr }M^n \leq c \lambda^n$. 
Hence, the set of non-primitive closed paths of length $n$ has cardinality
$$\textup{Tr }M^n - n (\# \Gamma_n) = \sum_{d|n, d \neq n} d (\#\Gamma_d) \leq \frac{c n}{2} \lambda^{n/2}$$
and so
$$\lim_{n \to \infty}\frac{n (\# \Gamma_n)}{\lambda^n} = 1.$$
Then for any set $A \subseteq \mathcal{F}_n$ and any $n \geq 1$ sufficiently large 
(using the bijection provided by Lemma \ref{L:code} with $n$ large enough to avoid the exceptions)
$$\mu_n(A) =  \frac{\sum_{d | n} \#(A \cap \Gamma_d)}{\sum_{d | n} (\# \Gamma_d)} = \frac{\#(A \cap \Gamma_n) + O(n \lambda^{n/2})}{(\# \Gamma_n) +  O(n \lambda^{n/2})}$$
and, since the map $p_n$ is exactly $d$-to-$1$ on the preimage of $\Gamma_d$,  
$$\lambda_n(p_n^{-1}(A))= \frac{\sum_{d | n}  d \#(A \cap \Gamma_d)}{\sum_{d | n} d (\# \Gamma_d)} = \frac{n \#(A \cap \Gamma_n) + O(n \lambda^{n/2})}{n (\# \Gamma_n) + O(n \lambda^{n/2})}.$$ 
Hence
$$\left| \mu_n(A) - \lambda_n(p_n^{-1}(A)) \right| = O\left(\frac{n \lambda^{n/2}}{(\# \Gamma_n)}\right) \to 0, $$
as $n\to \infty$. This completes the proof.
\end{proof}

\subsection*{Markov chains} 

Let $\Gamma$ be a directed graph with vertex set $V = V(\Gamma)$, edge set $E = E(\Gamma)$, and aperiodic transition matrix $M$. 
For $n \ge0$, we let $\Omega^n$ denote the set of paths of length $n$ in $\Gamma$ starting at any vertex, and $\Omega^\star = \bigcup_{n \geq 1} \Omega^n$ 
the set of all paths of finite length. Also denote by $\Omega$ the set of all \emph{infinite} directed paths. If we wish to focus on the subset of paths that start at the vertex $v \in V$, then we use $v$ as a subscript as in $\Omega_v$ or $\Omega^n_v$.
We will associate to $\Gamma$ a probability on each edge and a probability on each vertex so that the corresponding measure on the space of infinite paths in $\Gamma$ is shift-invariant. 
(The reader is directed to \cite[Section 8]{Walters} for additional details of this standard construction.)
Recall that if $x = (x_i)_{i\ge 0} \in \Omega$ is a directed path in $\Gamma$, then the image of $x$ under the shift $T \colon \Omega \to \Omega$ is $Tx = (x_i)_{i\ge 1}$. The assumption that $M$ is aperiodic translates to the fact that the shift is topologically mixing.
For $x\in \Omega^n$, we continue to use the notation $T^ix$ to denote the image of $x$ under the shift for $i\le n$.

Recall that $\lambda >1$ is the unique (simple) eigenvalue of $M$ of maximum modulus.
Fix a right eigenvector $\textbf{v}$ for $M$ of eigenvalue $\lambda$ and a left eigenvector $\textbf{u}$ of the same eigenvalue, normalized so that 
$\textbf{u}^T \textbf{v} = 1$. 
Let us now define the measure $(\pi_i)_{i \in V}$ on the set of vertices of $\Gamma$ where $\pi_i = u_i v_i$, 
and for each edge in $\Gamma$ from $i$ to $j$ let us define the probability $q_{ij} = \frac{m_{ij} v_j}{\lambda v_i}$. 
This defines a Markov chain on $\Gamma$, where $\pi_i$ is the probability of starting at vertex $v_i$, and $q_{ij}$ 
is the probability of going from $v_i$ to $v_j$. 

This construction defines a shift invariant measure $\nu$ (the so-called Parry measure) on the set of infinite paths $\Omega$. Namely, let 
$C(i_0, i_1, \dots, i_k)$ denote the set of infinite paths which start with the path $v_{i_0} \to v_{i_1} \to \dots \to v_{i_k}$. 
This is a cylinder set 
of $\Omega$ and we set its measure to be 
$$\nu(C(i_0, i_1, \dots, i_k)) = \pi_{i_0} q_{i_0 i_1} q_{i_1 i_2} \dots q_{i_{k-1} i_k},$$
and this determines a shift invariant measure $\nu$ on $\Omega$.
(In fact, this defines the measure of maximal entropy for the shift $T \colon \Omega \to \Omega$ \cite[Theorem 8.10]{Walters}.)
Now for each $n$, let $\nu_n$ be the pushforward of $\nu$ with respect to the map $\Omega \to \Omega^n$ which takes an infinite path to its 
 prefix of length $n$.
The measure $\nu_n$ is the distribution of the $n^{th}$ step of the Markov chain whose initial distribution is $(\pi_i)_{i \in V}$, and $\nu_n$ is supported on the set of paths of length $n$. 

As before $\ev : E(\Gamma) \to G$ is the evaluation map which associates to each edge its label in $S \cup S^{-1} \subset G$. By concatenation, the map extends to a map $\ev : \Omega^\star \to G$
from the set of all finite paths to $G$. Hence, if $x \in \Omega^n$ is a path of length $n$ given as a sequence of vertices $(x_i)_{i=0}^{n}$ associated to the edge path $e_1, \ldots, e_n$, then $\ev(x) = \ev(e_1) \cdots \ev(e_n)$ in $G$. In particular, if $x$ is a single vertex (i.e. a path of length $0$) then $\ev(x) = 1$.


\subsection{The Central Limit Theorem for H\"older observables}
Here we briefly recall the classical CLT from Thermodynamic Formalism as we will need it. We closely follow Bowen \cite[Chapter 1]{Bowen}.

Let $T \colon \Omega \to \Omega$ be a topologically mixing shift of finite type and $\nu$ a shift-invariant Gibbs measure on $\Omega$. (For us, $\nu$ will always be the measure of maximal entropy  on the Markov chain $\Omega$, defined as above.) We have a metric $d_\Omega$ on the space $\Omega^\star \cup \Omega$ 
of all directed paths  defined by $d_\Omega(x,y) = 2^{-K}$ where $K$ is the largest nonnegative integer such that $x_i = y_i$ for $i < K$. 
We note that with this metric, $\Omega$ is the set of limit points of the discrete set $\Omega^\star$. Topologically, $\Omega$ is a Cantor set.

A function $f \colon \Omega \to \mathbb{R}$ is \emph{H\"older continuous} 
if for all  $x,y \in \Omega$, $|f(x)-f(y)| \le Cd_\Omega(x,y)^\epsilon$ for some $C, \epsilon >0$. 

The following result is a combination of Theorem 1.27 of \cite{Bowen} and the remark that follows it. We will use the notation $N_{a, b} :=\frac{1}{\sqrt{2 \pi}} \int_a^b e^{-\frac{x^2}{2}} \ dx$.

\begin{theorem} \label{T:CLT}
Suppose that $f \colon \Omega \to \mathbb{R}$ is H\"older and that there does not exist a function $u \colon \Omega \to \mathbb{R}$ such that
$f  = u - u \circ T + \int f d\nu$. 
Then there is a constant $\sigma>0$ such that for any $a<b$,
$$\nu \left( x \in \Omega \ : \ \frac{\sum_{i=0}^{n-1}f(T^ix) - n \int f d\nu }{\sigma \sqrt{n}}  \in [a, b] \right) \to N_{a, b}.$$
\end{theorem}

\subsection{Displacement and the CLT}
The main result of this section is the following central limit theorem for displacement along the Markov chain. 
First, fix a basepoint $z\in \mathbb{H}^2$.

\begin{theorem} \label{T:disp}
There exist constants $L > 0$ and $\sigma > 0$ such that for any $a, b \in \mathbb{R}$ with $a < b$ we have 
$$\nu_n \left( x \in \Omega^n \ : \ \frac{d(\ev(x)z, z) - n L }{\sigma \sqrt{n}}  \in [a, b] \right) \to N_{a, b}.$$
\end{theorem}

The proof of Theorem \ref{T:disp} requires the following setup, which approximately follows the discussion in Calegari \cite[Section 3.7]{Ca}.
 For $g \in G$ and generator $s\in S$, define $D_s F(g) = d(z,gz) - d(z,sgz)$.
For a \emph{finite} path $x \in \Omega^\star$, let $DF(x) = D_{s^{-1}}F(\e(x))$, where $s$ labels the first edge of $x$. This defines a function $DF \colon \Omega^\star \to \mathbb{R}$ on the set of all finite paths such that
\begin{align*}
DF(x) = d(z,\e(x)z) - d(z,\e(T x)z).
\end{align*}
In particular, if $x \in \Omega^n$, then we note that
\begin{align} \label{sum:1}
\sum_{i=0}^{n-1}DF(T^ix) = d(z,\e(x)z).
\end{align}

To apply the central limit theorem (Theorem \ref{T:CLT}), one needs to verify the H\"older continuity property of the observable. For this, we need the following: 


\begin{lemma}[\cite{PS}, Proposition 1; \cite{PS2}, Lemma 1 and Proposition 3] \label{L:Holder}
Let $X$ be a $\mathrm{CAT}(-k)$ metric space, with $k > 0$. Then there exist constants $C > 0$ and $\alpha > 1$ such that for any $h, g \in G$, any $s \in S$
\begin{align}\label{e:holder}
|DF_s(g) - DF_s(h)| \le C \alpha^{- (g, h)}.
\end{align}
\end{lemma}




Since the evaluation map $\ev \colon \Omega^\star \to G$ is geodesic, then for any $x, y \in \Omega^\star$
$$\alpha^{-(\ev(x), \ev(y))} \leq d_\Omega(x, y)^\eta$$
where $\eta = \log_2(\alpha)$. Hence, the function $DF$ on $\Omega^\star$ is H\"older continuous and therefore has a unique continuous extension to the H\"older function $DF \colon \Omega \to \mathbb{R}$ on the space of all \emph{infinite} paths $\Omega$. 

For $x \in \Omega$, define
\begin{align*}
F_n(x) = DF(x) + DF(Tx) + \ldots +DF(T^{n-1}x).
\end{align*}
If we let $x^n \in \Omega^n$ denote the prefix of $x$ of length $n$, we have that 
\begin{align*}
F_n(x^n) = d(z,\e(x^n)z),
\end{align*}
by (\ref{sum:1}). Whereas for an infinite path $x \in \Omega$, we have 
\begin{align*}
F_n(x) = \lim_{k \to \infty} \big (d(z,\e(x^k)z) - d(z,\e(T^n x^k)z) \big).
\end{align*}
However, the following lemma bounds how far $F_n(x)$ can be from the displacement $F_n(x^n) = d(z,\ev(x^n)z)$.

\begin{lemma}\label{lem:almost_displacement}
The difference $| F_n(x) - F_n(x^n)|$ is uniformly bounded, independent of $x \in \Omega$ and $n \ge 1$.
\end{lemma}
\begin{proof}
Write $x = y_1y_2 \ldots$, where the $y_i$ are edges of $\Gamma$, i.e. we represent $x$ as an edge path in $\Gamma$. Then $x^k = y_1y_2\ldots y_k$. Also, let $\e(y_i) = g_i$. Then 
\begin{align*}
|F_n(x^n) - F_n(x)| &=  \lim_{k\to \infty} \big( d(z,g_1\ldots g_n z) + d(z, g_{n+1}\ldots g_k z) -  d(z, g_1\ldots g_k z) \big) \\
& = \lim_{k\to \infty} \big( d(z,g_1\ldots g_n z) + d(g_1\ldots g_n z, g_{1}\ldots g_k z) -  d(z, g_1\ldots g_k z) \big) \\
&= 2 \lim_{k \to \infty} \big(z,\e(x^k)z \big)_{\e(x^n)z}.
\end{align*}

Since $G \curvearrowright \mathbb{H}^2$ is convex cocompact and the path $i \to \e(x^i) = g_1\ldots g_i$ is geodesic in $G$, the path $i \to \e(x^i) z$ is a uniform quasigeodesic. 
That is,  $i \to \e(x^i) z$ is a $(K,C)$--quasigeodesic in $X$ for $K\ge 1$ and $C\ge 0$ not depending on $x$. This (together with the stability of quasigeodesics in the hyperbolic space $\mathbb{H}^2$) immediately implies that the quantity $\big(z,\e(x^k)z \big)_{\e(x^n)z}$ is uniformly bounded for every $k \ge n$. This completes the proof.
\end{proof}

Finally, the following lemma is needed to establish positivity of $\sigma$ in Theorem \ref{T:disp}.
\begin{lemma} \label{lem:sig_pos}
There does not exist a function $u : \Omega \to \mathbb{R}$ and $L \in \mathbb{R}$ such that 
$$DF = u - u \circ T + L.$$
\end{lemma}

\begin{proof}
Suppose not. Then 
$$\sum_{i = 0}^{n-1} DF(T^i(x)) = u(x) - u(T^n(x)) + nL$$
Thus, if $x \in \Omega$ is a periodic point of period $n$ for $T$, then
$$F_n(x) = \sum_{i = 0}^{n-1} DF(T^i(x)) = n L .$$
Next, a direct computation 
shows that if $g = \ev(x^n)$, then
\[
\tau(g) = \tau(\ev(x^n)) = F_n(x) = nL.
\]
Indeed, as in \cite[Lemma 4.2]{HS}, since $T^nx = x$, $F_{nk}(x) = k F_n(x)$ and we see that $F_{nk}(x) = F_{nk}(x^{nk}) + O(1)$ (Lemma \ref{lem:almost_displacement}) implies 
\[
F_n(x) = \lim_{k \to \infty} \frac{1}{k}F_{nk}(x^{nk}) = \lim_{k\to \infty} \frac{1}{k}d(z,g^kz),
\]
which, by definition, equals the translation length $\tau(g)$.

Hence, we conclude using Lemma \ref{L:code} that $\tau([g]) = L \Vert g \Vert$ for all but finitely many conjugacy classes in $G$. This, however, contradicts the fact that there exists two closed geodesics on $\Sigma$ with incommensurable lengths (see, for example, \cite{Dal'bo}). 
\end{proof}

\begin{proof}[Proof of Theorem \ref{T:disp}]
Since $DF$ is a H\"older continuous function on a mixing shift of finite type, 
Theorem \ref{T:CLT} along with Lemma \ref{lem:sig_pos}
give 
\[
\nu \left( x \ : \ \frac{F_n(x) - n L }{\sigma \sqrt{n}}  \in [a, b] \right) \to N_{a, b},
\]
where as before $F_n(x) = DF(x) + DF(Tx) + \ldots +DF(T^{n-1}x)$. 
But by Lemma \ref{lem:almost_displacement}, the probability 
\[
\nu \left(x \ : \frac{|F_n(x^n) - F_n(x)|}{\sqrt{n}} \ge \epsilon \right) \to 0,
\]
as $n \to \infty$, for any $\epsilon > 0$. 
Hence, applying Lemma \ref{L:sum} gives the CLT for the displacement $F_n(x^n) = d(z,\e(x^n)z)$ for $x \in \Omega$ with respect to the measure $\nu$. Since the distribution of $x^n \in \Omega^n$ is $\nu_n$, this completes the proof of the theorem.
\end{proof}

\section{Convergence to the counting measure for closed paths}
Next we use the Theorem \ref{T:disp} (which is about the Markov chain) to study the distribution of closed paths. 
Given a path $x$ of length $n$, let $\widehat{x}$ denote the prefix of $x$ of length $n - \log n$. 

Recall that $\lambda_n$ is the uniform distribution on the set $\mathcal{C}_n$ of based closed paths of length $n$ in $\Gamma$. Since $\mathcal{C}_n \subset \Omega^n$, $\lambda_n$ defines a measure on $\Omega^n$ supported on $\mathcal{C}_n$.
Let $\lambda_{n, m}$ denote the distribution of the prefix of length $n - m$ of a uniformly chosen closed path of length $n$. Said differently, $\lambda_{n,m}$ is the distribution on paths of length $n-m$ obtained by pushing $\lambda_n$ forward under the prefix map.
In particular, if $x$ has law $\lambda_n$, then $\widehat{x}$ has law $\lambda_{n, \log n}$. We define $\nu_{n,m}$ in the same way, using $\nu_n$ in place of $\lambda_n$. By the Markov property, $\nu_{n,m} = \nu_{n-m}$.

\begin{proposition} \label{P:RN-der}
With notation as above, $\lambda_{n, m}$ is absolutely continuous with respect to $\nu_{n, m}$, and moreover 
$$\sup_{\gamma \in \Omega^{n-m}} \left| \frac{d \lambda_{n, m}}{d \nu_{n, m}}(\gamma) - 1 \right| \to 0$$
as $\min \{ m, n \} \to \infty$.
\end{proposition}

\begin{proof}
Given a path $\gamma = e_1 \cdot \ldots \cdot e_{n-m}$ with starting vertex $v_i$ and end vertex $v_j$ we have 
$$\lambda_{n, m}(\gamma) = \frac{\#\{\textup{paths of length }m\textup{ from }v_j\textup{ to }v_i\}}{\# \{\textup{closed paths of length }n \}}$$
$$ = \frac{ \textbf{e}_j^T M^m \textbf{e}_i }{\textup{Tr }M^n}$$
where $\textbf{e}_i$ is the $i^{th}$ basis vector.

Now, recall that we have fixed a right eigenvector $\textbf{v}$ for $M$ of eigenvalue $\lambda$, and a left eigenvector $\textbf{u}$ of the same eigenvalue, normalized so that 
$\textbf{u}^T \textbf{v} = 1$.
Since $M$ is irreducible and aperiodic, by the Perron-Frobenius Theorem we have 
$$\lim_{n \to \infty} \frac{M^n}{\lambda^n} = \textbf{v} \textbf{u}^T$$
and in particular 
$$\lim_{n \to \infty} \frac{\textbf{e}_i^T M^n \textbf{e}_j}{\lambda^n} = \textbf{e}_i^T  \textbf{v} \textbf{u}^T \textbf{e}_j = v_i u_j.$$

As before the measure $(\pi_i)$ where $\pi_i = u_i v_i$ is stationary for the Markov chain defined as $q_{ij} = \frac{m_{ij} v_j}{\lambda v_i}$, so 
we consider this Markov chain with the stationary measure as starting distribution.
Let $\nu_{n, m} = \nu_{n-m}$ be the pushforward of the Markov measure on the set of paths of length $n-m$. Then 
$$\nu_{n, m}(\gamma) = \frac{\pi_i v_j}{v_i \lambda^{n-m}} = \frac{u_i v_j}{\lambda^{n-m}}.$$
Hence 
$$\frac{d \lambda_{n, m}}{d \nu_{n, m}}(\gamma) = \frac{ \textbf{e}_j^T M^m \textbf{e}_i }{\textup{Tr }M^n} \frac{\lambda^{n-m}}{u_i v_j} = \frac{ \textbf{e}_j^T M^m \textbf{e}_i }{\lambda^m u_i v_j} \frac{\lambda^n}{\textup{Tr }M^n} \to 1,$$
as $\min\{m,n\} \to \infty$.
\end{proof}

We conclude this section by promoting the CLT for displacement from the Markov chain (Theorem \ref{T:disp}) to the counting measure $\lambda_n$.

\begin{theorem} \label{T:disp2}
For any $a, b \in \mathbb{R}$ with $a < b$ one has
$$\lambda_n \left( x \ : \ \frac{d(\ev(x)z, z) - n L }{\sigma \sqrt{n}}  \in [a, b] \right) \to N_{a, b}$$
as $n \to \infty$.
\end{theorem}

\begin{proof}
Since $\log n/\sqrt{n} \to 0$, Theorem \ref{T:disp} (together with Lemma \ref{L:sum}) implies
$$\nu_{n-\log n} \left( x  \ : \ \frac{d(\ev(x)z, z) - n L }{\sigma \sqrt{n}}  \in [a, b] \right) \to N_{a, b}.$$
Now, since $\nu_{n - \log n} = \nu_{n, \log n}$ and by Proposition \ref{P:RN-der} and Lemma \ref{L:TV} we get 
$$\Vert \nu_{n, \log n} - \lambda_{n, \log n} \Vert_{TV} \to 0$$
hence
$$\lambda_{n, \log n} \left( x \ : \ \frac{d(\ev(x)z, z) - n L }{\sigma \sqrt{n}}  \in [a, b] \right) \to N_{a, b}.$$
Moreover, by the definition of $\widehat{x}$,
$$\lambda_n \left( x \ : \ \frac{d(\ev(\widehat{x})z, z) - n L }{\sigma \sqrt{n}}  \in [a, b] \right) = \lambda_{n, \log n} \left( x \ : \ \frac{d(\ev(x)z, z) - n L }{\sigma \sqrt{n}}  \in [a, b] \right). $$
Finally, note the since the orbit map $G \to \mathbb{H}^2$ is Lipschitz, we have 
$$|d(\ev(x)z, z) - d(\ev(\widehat{x}) z, z)| \leq C \log n$$ 
where $C$ is the Lipschitz constant. Then using Lemma \ref{L:sum}
\begin{align*}
 \lim_{n \to \infty} \lambda_n \left( x \ : \ \frac{d(\ev(x) z, z) - n L }{\sigma \sqrt{n}}  \in [a, b] \right) 
  &= \lim_{n \to \infty} \lambda_n \left( x \ : \ \frac{d(\ev(\widehat{x}) z, z) - n L }{\sigma \sqrt{n}}  \in [a, b] \right)  \\
  &= N_{a, b}
\end{align*}
which completes the proof.
\end{proof}



\section{The Gromov product}

The remaining step of our proof is to turn the statement about displacement (Theorem \ref{T:disp2}) into a statement about translation length. This is done by controlling the Gromov product.

We begin with the following easy computation. Recall that $\Omega^n$ is the set of all paths of length $n$ and $\mathcal{C}_n \subset \Omega^n$ is the subset of closed paths.

\begin{lemma} \label{L:Rn}
There is a constant $D \ge0$ such that 
\[
1 \le \frac {\# (\Omega^n)} {\# (\mathcal{C}_n)} \le D
\]
\end{lemma}

\begin{proof}
We know that $\# (\mathcal{C}_n) = \textup{Tr }M^n = \lambda^n + \sum_i {\lambda_i^n}$, where $\lambda$ is the Perron--Frobenius eigenvalue of $M$ and $|\lambda_i|< \lambda$.
Also, $\# (\Omega^n) = \Vert M^n \Vert_1 \le D \lambda^n$ for some $D \ge 1$.
\end{proof}

Next, we will see that the Gromov product of a random element and its inverse grows slowly in word length. This is our key estimate for relating displacement to translation length. 
\begin{proposition} \label{P:gr}
For any $\epsilon >0$, 
$$\lambda_n \big ( x \ : \ (\ev(x)z, \ev(x)^{-1}z)_z \leq \epsilon \sqrt{n} \big )\to 1,$$
as $n \to \infty$.
\end{proposition}

\begin{proof}
Fix $\epsilon > 0$. Define 
$$A =   \big \{ x \in \Omega^\star \ : \ (\ev(x)z, \ev(x)^{-1}z)_z \geq \epsilon \sqrt{|x|} \big \},$$
where $|x|$ denotes the length of $x$. By \cite[Section 6]{GTT} for any vertex $v$ of $\Gamma$ one has  
$$\mathbb{P}_v\left( (\ev(x_n) z, \ev(x_n)^{-1}z)_z \geq \epsilon \sqrt{n} \right )\to 0,$$
where $\mathbb{P}_v$ is the distribution of the Markov chain starting at vertex $v$, and $x_n$ is the $n^{th}$ step. (Lemma 6.2 of \cite{GTT} explicitly gives this statement where $v$ is the ``initial vertex'' of the directed graph, however the same argument gives the more general result for all vertices.)

Next, consider any path $p \in \Omega_v^n$ of length $n$ starting at $v$. If $v$ is the $l^{th}$ vertex of $\Gamma$ and the terminal endpoint of $p$ is the $k^{th}$ vertex, then the $n^{th}$ step measure of $p$ is
$
\mathbb{P}^n_v(p) = \frac{1}{\lambda^n} \frac{v_k}{v_l}.
$
(Here, $\mathbb{P}^n_v$ is the measure supported on $\Omega_v^n$ which is the pushforward of $\mathbb{P}_v$ under the prefix map $\Omega_v \to \Omega_v^n$.)
Hence, as in \cite[Lemma 3.4]{GTT}),  there is a constant $c>1$, depending only on the adjacency matrix $M$, such that for any subset $B \subset \bigcup_{n\ge1}\Omega_v^n$
\[
\frac{1}{c} \cdot \frac{\# (B \cap \Omega_v^n)}{\# \Omega_v^n}  \le \mathbb{P}^n_v(B)\le c \cdot \frac{\# (B \cap \Omega_v^n)}{\# \Omega_v^n}. 
\]
In particular, we conclude that
\begin{align} \label{e:growth}
\frac{\# (A \cap \Omega_v^n)}{\# \Omega_v^n} \leq c \cdot \mathbb{P}_v \left( (\ev(x_n) z, \ev(x_n)^{-1}z)_z \geq \epsilon \sqrt{n} \right ) \to 0,
\end{align}
as $n\to \infty$.

Hence, by summing over all vertices $v \in V$
$$\frac{\# (A\cap \Omega^n)}{\# \Omega^n} \to 0.$$
Then by combining this with Lemma \ref{L:Rn},
$$\lambda_n(A) = \frac{\# (A \cap \mathcal{C}_n) }{\# \mathcal{C}_n} \leq \frac{\# (A \cap \Omega^n)}{\# \Omega^n} \cdot \frac{\# \Omega^n}{\# \mathcal{C}_n} \to 0.$$
\end{proof}

\begin{remark}
The main estimate (\ref{e:growth}) in the proof of Proposition \ref{P:gr} can also be obtained via a trick using more recent work of the authors \cite{GTT2}. Using terminology there, one may define a geodesic graph structure $(G,\Gamma)$ by declaring that $v$ be the initial vertex of $\Gamma$. Such a structure may not be surjective, but this is not necessary. Then \cite[Proposition 5.8]{GTT2} precisely gives the required decay result.
\end{remark}

\subsection*{Proof of Theorem \ref{T:main}}
We can now complete the proof of Theorem \ref{T:main}.

\begin{proof}[Proof of Theorem \ref{T:main}]
Let $d_n := \frac{d(z, \ev(x) z) - n L }{\sigma \sqrt{n}}$, $t_n := \frac{\tau( \ev(x) ) - n L }{\sigma \sqrt{n}}$, and $p_n := t_n - d_n = \frac{2 (\ev(x) z, \ev(x)^{-1} z)_z + O(\delta)}{\sigma \sqrt{n}}$.
By the CLT for displacement (Theorem \ref{T:disp}), for any $a < b$  
$$\lambda_n(x : d_n \in [a, b]) \to N_{a, b}.$$ 
Moreover, by decay of Gromov products for any $\epsilon > 0$  (Proposition \ref{P:gr}) we have
$$\lambda_n(x: |p_n| \geq \epsilon) \to 0$$
hence by Lemma \ref{L:sum} 
$$\lambda_n \left( x \ : \ \frac{\tau(\ev(x)) - n L }{\sigma \sqrt{n}} \in [a, b] \right) \to N_{a, b}.$$
Finally, by Lemma \ref{L:push} this implies 
$$\mu_n \left( \gamma \ : \ \frac{\tau(\gamma) - n L }{\sigma \sqrt{n}} \in [a, b] \right) \to N_{a, b}$$
which completes the proof.
\end{proof}





\section{Generalizations}
While many generalization of Theorem \ref{T:main} are possible, we record the most immediate one here. The proof is the same as the one given above.

\begin{theorem}
Let $G\curvearrowright X$ be any convex cocompact action of a closed orientable surface group on a $\mathrm{CAT}(-1)$ space.
Then the conclusion of Theorem \ref{T:main} holds. The same is true for any convex cocompact free group action on a $\mathrm{CAT}(-1)$ space  $X$ as long as the lengths of closed geodesics on $X/G$ are not all contained in $c\Z$ for some $c>0$.
\end{theorem}

The condition on lengths of closed geodesics is known to hold whenever $G\curvearrowright X$ is a discrete action on a $\mathrm{CAT}(-1)$ satisfying at least one of the following:
 \begin{itemize}
\item the limit set $\Lambda(G)\subset \partial X$ has an infinite connected component (in particular when $G$ is a surface group) \cite{Bourdon};
\item $X$ is itself a surface (so that $X/G$ is a locally $\mathrm{CAT}(-1)$ surface) \cite{Dal'bo}; 
\item $X$ is a rank 1 symmetric space \cite{Kim}.
\end{itemize}

\end{document}